\documentclass[12pt,reqno]{amsart}
\usepackage{amssymb,mathrsfs}

\def\abs#1{|#1|}

\def\norm#1{\|#1\|}

\def\wh#1{\widehat{#1}}
\newcommand{\del}{\ensuremath{D^{\alpha}}}

\newcommand{\cS}{\mathcal{S}}

\newcommand{\K}{\mathcal{K}}

\newcommand{\R}{\mathbb{R}}

\newcommand{\al}{\alpha}

\newcommand{\e}{\varepsilon}

\newcommand{\ft}{{\mathcal{F}}}

\newcommand{\re}{\mathop{\mathrm{Re}}}
\newcommand{\im}{\mathop{\mathrm{Im}}}

\newcommand{\ges}{{\gtrsim}}

\newcommand{\les}{\lesssim}

\newcommand{\I}{\infty}

\newcommand{\paren}[1]{\ensuremath{\left( #1\right)}}
\newcommand{\abso}[1]{\ensuremath{\left| #1\right|}}

\newcommand{\EQ}[1]{\begin{equation}\begin{split} #1 \end{split}\end{equation}}
\newcommand{\Del}[1]{}

\numberwithin{equation}{section}

\newtheorem{lemma}[equation]{Lemma}
\newtheorem{thm}{Theorem}[section]

\newtheorem{lem}[thm]{Lemma}

\theoremstyle{remark}
\newtheorem{rem}{Remark}
\newtheorem{defn}{Definition}[section]

\begin{document}
%\subjclass[2010]{35L70, 35Q55}
%\keywords{Nonlinear wave equation, nonlinear Schr\"odinger equation}

\title[fractional Schr\"odinger equation]{On the energy-critical fractional
Sch\"odinger equation in the radial case}

\author[Z. Guo]{Zihua Guo}
\author[Y. Sire]{Yannick Sire}
\author[Y. Wang]{Yuzhao Wang}
\author[L. Zhao]{Lifeng Zhao}

\address{Z. Guo, School of Mathematical Sciences, Peking
University, Beijing 100871, China} \email{zihuaguo@math.pku.edu.cn}

\address{Y. Sire, Universit\'{e} Aix-Marseille and LATP 9, rue F. Joliot Curie, 13453
Marseille Cedex 13, France} \email{sire@cmi.univ-mrs.fr}

\address{Y. Wang, Department of Mathematics and Physics, North China Electric Power University, Beijing 102206,
China} \email{wangyuzhao2008@gmail.com}

\address{L. Zhao, University of Science and Technology of China, Hefei, China} \email{zhaolf@ustc.edu.cn}

\begin{abstract}
We consider the Cauchy problem for the energy-critical nonlinear
Schr\"odinger equation with fractional Laplacian (fNLS) in the
radial case. We obtain global well-posedness and scattering in the
energy space in the defocusing case, and in the focusing case with
energy below the ground state.
\end{abstract}

\maketitle

%\tableofcontents

\section{Introduction}

In this paper, we study the Cauchy problem for the nonlinear
Schr\"odinger equation with fractional Laplacian:
\begin{equation}\label{eq:fNLS}
\begin{cases}
    i\partial_tu + D^{2\alpha} u +\mu |u|^{\frac{4\alpha}{N-2\alpha}}u = 0 &(x,t)\in \R^N\times \R \\
    u|_{t=0} = u_0\in \dot{H}^{\alpha}(\R^N),
\end{cases}
\end{equation}
where $\alpha\in (\frac{N}{2N-1},1), D=\sqrt{-\Delta}$, $\mu\in
\{-1,1\}$. Here $\mu=1$ corresponds to the defocusing case, and
$\mu=-1$ corresponds to the focusing case. When $\alpha=1$,
\eqref{eq:fNLS} is the well-known energy-critical nonlinear
Schr\"odinger equation which has been extensively studied, and we
refer the readers to \cite{Visan} for a survey of the study. When
$0<\alpha<1$, \eqref{eq:fNLS} is a nonlocal model known as nonlinear
fractional Schr\"odinger equation which has also attracted much
attentions recently (see
\cite{GHX,GH,GH2,CHKL1,CHKL2,GH3,CHHO,CHHO2}). The fractional
Schr\"odinger equation is a fundamental equation of fractional
quantum mechanics, which was derived by Laskin
\cite{Laskin1,Laskin2} as a result of extending the Feynman path
integral, from the Brownian-like to L\'{e}vy-like quantum mechanical
paths. The purpose of this paper is to prove some analogue global
well-posedness and scattering for \eqref{eq:fNLS}  in the radial
case.

Under the flow of the equation \eqref{eq:fNLS}, the following
quantities (mass and energy) are conserved:
\begin{align*}
M(u)=&\int_{\R^N} |u(x,t)|^2dx,\\
E_\mu(u)=&\int_{\R^N}\frac{1}{2}{|D^\alpha u|^2}+\frac{\mu
}{p+2}|u|^{p+2}dx.
\end{align*}
We write $E_\pm (u)=E_{\pm 1}(u)$. Moreover, the equation
\eqref{eq:fNLS} preserves the radial symmetry, and also has the
following scaling invariance: for $\lambda>0$
\[u(x,t)\to \lambda^{\frac{N-2\alpha}{2}} u(\lambda x,\lambda^{2\alpha} t), \quad u_0(x)\to \lambda^{\frac{N-2\alpha}{2}} u_0(\lambda x).\]
Thus, \eqref{eq:fNLS} is $\dot{H}^{\alpha}$-critical,
since the scaling transform leaves $\dot{H}^{\alpha}$-norm invariant.

There are remarkable differences between the defocusing and focusing
cases. In the focusing case, the flow has more kinds of dynamical
behavior. An important role is played by the ground state $W_\al$,
namely the unique non-negative radial solution to the fractional
elliptic equation
\begin{equation}\label{eq:fell}
(-\Delta)^{\alpha} W - \abs{W}^{\frac{4\alpha}{N-2\alpha}}W = 0.
\end{equation}
We have $W_\al\in \dot H^\alpha$, and so $W_\al$ is a stationary
solution to \eqref{eq:fNLS} when $\mu=-1$. See section 3 for more
properties of $W_\al$. The main result of this paper is

\begin{thm}\label{thm}
Assume $N\geq 2$, $\alpha\in (\frac{N}{2N-1},1)$, $2\alpha < N <
6\alpha$, $W_\al$ as above. Assume $u_0\in \dot{H}^\alpha$, $u_0$
radial. Then

\begin{enumerate}
\item Defousing case ($\mu=1$): \eqref{eq:fNLS} is globally
well-posed, and scattering holds.

\item Focusing case ($\mu=-1$): if $E_-(u_0)<E_-(W_\al)$ and $\norm{\del
u_0}_2<\norm{\del W_\al}_2$, then \eqref{eq:fNLS} is globally
well-posed, and scattering holds.
\end{enumerate}
\end{thm}

Now we discuss the ideas of proof. We follow closely the
Kenig-Merle's concentration compactness/rigidity method \cite{KM}.
There are several different ingredients:

\begin{enumerate}
  \item Radial Strichartz estimates. When $\alpha<1$, we know that the classical Strichartz estimates in non-radial case has loss of regularity.
  However, in the radial case, it was known that when $\alpha\in (\frac{N}{2N-1},1)$ one has generalized estimates which has no loss of derivatives, see \cite{GW10}. In contrast to \cite{KM}, radial symmetry for \eqref{eq:fNLS} plays crucial role in many aspects.

  \item The results from the study of the fractional elliptic equation.
  The fractional elliptic equation has been extensively studied recently.
  In the focusing case, we will apply the results for \eqref{eq:fell} which was obtained in \cite{Lieb}, \cite{CLO}.

  \item Localization of virial identity. In the rigidity argument, we use the localization of virial identity.
  Due to the nonlocal nature of $(-\Delta)^\alpha$, we need to deal with some commutator estimates.
\end{enumerate}
The main difference between \eqref{eq:fNLS} and Schr\"odinger
equation is the nonlocal property of the fractional Laplacian. In
our proof, this nonlocal property makes only slight difference from
the Kenig-Merle's argument in the concentration-compactness part
(Thus we omit most of the details). However, it makes big difference
in the space-time a-priori estimates, e.g. localization of virial
estimates in the rigidity part. We do not know any other monotonity,
such as Morawetz estimates.

\section{The Cauchy problem and the variational estimates}

\subsection{The Cauchy problem}
In this section, we review the local theory and small data global
theory for the Cauchy problem \eqref{eq:fNLS} with radial symmetry.
It has no difference between defocusing and focusing cases. The key
ingredient is the radial Strichartz estimates obtained in
\cite{GW10}.

%\begin{defn}
%Suppose $N\ge 2$. The exponent pair $(q,r)$ is said to be $n-D$ %radial Schr\"odinger admissible if $q,r \ge 2$, and
%\begin{align}\label{n-D-admissible}
%&\frac{4N+2}{2N-1} \le q \le \infty, \quad \frac2q + \frac{2N-1}{r} %\le n-\frac12\\
%or &\nonumber\\
%&2\le q< \frac{4N+2}{2N-1}, \quad
%\end{align}
%\end{defn}

\begin{lem}[Proposition 3.9 \cite{GW10}]\label{prop:StriFSch}
Suppose $N\geq 2$, $\alpha>1/2$ and $u,u_0, F$ are spherically
symmetric in space and satisfy
\begin{equation*}
\begin{cases}
    i\partial_tu + (-\Delta)^{\alpha} u =F   &(x,t)\in \R^N\times \R \\
    u|_{t=0} = u_0.
\end{cases}
\end{equation*}
Then for $\gamma \in \R$ it holds
\begin{align}
\norm{u}_{L_t^qL_x^r}+\norm{u}_{C(\R:\dot{H}^\gamma)}\les
\norm{u_0}_{\dot{H}^\gamma}+\norm{F}_{L_t^{\tilde q'}L_x^{\tilde
r'}},
\end{align}
if the following conditions hold:

(1) $(q,r)$ and $(\tilde q,\tilde r)$ both satisfy the following conditions:
\begin{align}\label{radial-admis}
2\leq q,r\leq \infty,\frac{1}{q}<
(N-\frac{1}{2})(\frac{1}{2}-\frac{1}{r});
\end{align}

(2) $\tilde q'<q$ and the ``gap"
condition:
\[\frac{2\alpha}{q}+\frac{N}{r}=\frac{N}{2}-\gamma,\ \frac{2\alpha}{\tilde q}+\frac{N}{\tilde r}=\frac{N}{2}+\gamma.\]
\end{lem}

\begin{rem}
The conditions in (1) can be relaxed to the following
\begin{align}
2\leq q,r\leq \infty,\frac{1}{q}\leq
(N-\frac{1}{2})(\frac{1}{2}-\frac{1}{r}),\quad (q,r)\ne
(2,\frac{4N-2}{2N-3}).
\end{align}
On the boundary line $\frac{1}{q}=
(N-\frac{1}{2})(\frac{1}{2}-\frac{1}{r})$, \cite{GW10} first proved
it for $q\geq r$, and was later improved to other pairs
independently by \cite{Ke} and \cite{CL}.
\end{rem}

\begin{defn}
For $N\ge 2$, we say that a pair of exponents $(q,r)$ is
$\alpha$-admissible if $(q,r)$ verifies
\begin{align}\label{admissible}
\frac {2\alpha}q + \frac Nr = \frac N2,\quad 2\leq q,r \leq \infty.
\end{align}
\end{defn}

By Lemma \ref{prop:StriFSch}, we see that if $\al\in
(\frac{N}{2N-1},1)$, then we have a full set of $\alpha$-admissible
Strichartz estimates which has no loss of derivatives. With these
Strichartz estimates, we can proceed as the classical theory of
Schr\"odinger equation. Let $I\subset \R$ be an interval, and we
define $S_\al(I), W_\al(I)$ norm by
\[\norm{v}_{S_\al(I)} = \norm{v}_{L_I^{\frac{2(N+2\alpha)}{N-2\alpha}} L_x^{\frac{2(N+2\alpha)} {N-2\alpha}}} \  \ \mbox{and}  \   \
\norm{v}_{W_\al(I)} = \norm{v}_{L_I^{\frac{2(N+2\alpha)}{N-2\alpha}}
L_x^{\frac{2N(N+2\alpha)}{N^2+ 4\alpha^2}}}.\] Note that
$(\frac{2(N+2\alpha)}{N-2\alpha},
\frac{2N(N+2\alpha)}{N^2+4\alpha^2})$ is $\alpha$-admissible pairs.
By Sobolev embedding, we have that if $N >2\alpha$,
\[ \norm{v}_{S_\al(I)}\le C\norm{D^{\alpha}v}_{W_\al(I)}.\]

\begin{defn}
Let $t_0\in I$.  We say that $u\in C(I;\dot{H}^{\alpha}(\R^N))\cap
\{D^\alpha u\in W_\al(I)\}$ is a solution of the \eqref{eq:fNLS} if
$$u|_{t_0} = u_0, \mbox{ \ \ and \ \ } u(t) = e^{i(t-t_0)(-\Delta)^{\alpha}}u_0 + \int_{t_0}^te^{i(t-t')(-\Delta)^{\alpha}}\abs{u}^{\frac{4\alpha}{N-2\alpha}}u\,dt'.$$
\end{defn}

\begin{defn}
Let $v_0\in \dot{H}^{\alpha}$, $v(x,t) = e^{it(-\Delta)^\al}v_0$ and
let $\{t_n\}$ be a sequence, with $\lim_{n\to \infty}t_n =
\overline{t}\in [-\infty,+\infty]$.  We say that $u(x,t)$ is a
non-linear profile associated with $(v_0,\{t_n\})$ if there exists
an interval $I$, with $\overline{t}\in I$ (if $\overline{t} = \pm
\infty$, $I = [a,+\infty)$ or $(-\infty, a]$) such that $u$ is a
solution of (CP) in $I$ and
\[
    \lim_{n\to \infty}\norm{ u(-,t_n)-v(-,t_n)}_{\dot{H}^1} = 0.
\]
\end{defn}

With the Strichartz estimates, we can obtain the following results
for \eqref{eq:fNLS} by standard arguments (for example, see
\cite{CW}).

\begin{thm}

(1) Assume $N\geq 2$, $\alpha\in (\frac{N}{2N-1},1)$, $2\alpha < N <
6\alpha$ and $u_0\in \dot{H}^{\alpha}(\R^N)$, $u_0$ radial,
$\norm{u_0}_{\dot H^\al}\leq A$. Then $\exists \delta=\delta(A)$
s.t. if $\norm{e^{it(-\Delta)^\alpha} u_0}_{S_\al(I)}\leq \delta$,
$0\in \dot I$, there exists a unique solution to \eqref{eq:fNLS} on
$I$ such that $u\in C(I;\dot H^\al)$, $\sup_{t\in
I}\norm{u(t)}_{\dot H^\al}+\norm{D^\al u}_{W_\al(I)}\leq C(A)$ and
$\norm{u}_{S_\al(I)}\leq 2\delta$. Moreover, we have
\begin{itemize}
\item Local existence: there exists a maximal open interval $I=(-T_-(u_0),T_+(u_0))$ where the solution $u$
is defined.

\item Small data global existence: if $A\ll 1$, then
$I=(-\infty,+\infty)$.

\item $D^\alpha u\in L_{t}^qL_x^r(I'\times \R^N)$ for any $\alpha$-admissible pair
$(q,r)$, where $I'\subset I$ is a closed interval with finite
length.

\item Blowup criterion: If $T_+(u_0)<+\infty$, then $\|u\|_{S_\alpha([0,T_+(u_0)))} = +\infty$.
A similar statement holds in the negative time direction.

\item Scattering: If $T_+(u_0) = + \infty$ and $u$ dose not blow up forward in time, then $u$ scatters forward in time,
that is, there exists a unique $u_+\in \dot H^{\alpha}$ such that
\begin{align}\label{scattering}
\lim_{t\rightarrow +\infty}\|u(t)-
e^{it(-\Delta)^{\alpha}}u_+\|_{H^{\alpha}(\R^N)} = 0.
\end{align}A similar
statement holds in the negative time direction.
\end{itemize}

(2) For any $u_+\in \dot H^\al$, there exists a solution $u$ to
\eqref{eq:fNLS} such that \eqref{scattering} holds. As a
consequence, for any $(v_0,\{t_n\})$, there always exists a
non-linear profile associated to $(v_0,\{t_n\})$ with a maximal
interval of existence.
\end{thm}

Next, we need a perturbation theorem. It follows in a very similar
way as Theorem 2.14 in \cite{KM} (see \cite{Kenignotes} for a
correct proof), see also \cite{TV}. Since for $\alpha\in
(\frac{N}{2N-1},1)$, we have generalized inhomogeneous Strichartz
estimate given by Lemma \ref{prop:StriFSch}, the proof is with
slight change and we omit the details,

\begin{thm}[Stability]\label{stability}
Assume $N\geq 2$, $\alpha\in (\frac{N}{2N-1},1)$, $2\alpha < N <
6\alpha$. Let $I = [0,L)$, $L\leq +\infty$, and let $\tilde u$ be
defined on $I\times \R^N$ such that
\begin{align}
\|\tilde u\|_{L_t^{\infty} \dot H^{\alpha}_x(I\times \R^N)} \le A,
\quad \|\tilde u\|_{S_\al(I)} \le M, \quad \|D^{\alpha}\tilde
u\|_{W_\al(I)} <\infty
\end{align}
for some constants $A$ and $M$, and $\tilde u$ verifies in the sense
of integral equation
\[
i\tilde u_t + (-\Delta)^{\alpha} \tilde u +\mu |\tilde
u|^{\frac{4\alpha}{N-2\alpha}}\tilde u = e
\]
for some function $e$. Let $u_0\in \dot H^\al$ be such that
$\|u(0)-\tilde u(0)\|_{\dot H^{\alpha}} \le A'$. Then $\exists
\e_0=\e_0(M,A,A')$ s.t. if $0<\e<\e_0$ and
\begin{align}\label{sma-per-data}
\|e^{it(-\Delta)^{\alpha}} (u(0)-\tilde u(0))\|_{S_\al(I)} \le
\varepsilon,\quad \|D^{\alpha} e\|_{L^2_IL^{\frac
{2N}{N+2\alpha}}_x} \le \varepsilon,
\end{align}
then, $\exists\ !$ solution $u$ on $I\times \R^N$ to \eqref{eq:fNLS}
with initial data $u_0$ satisfying
\begin{align*}
\|u\|_{S_\al(I)} \leq C(A,A',M),\quad \sup_{t\in I}\norm{u(t)-\tilde
u(t)}_{\dot H^\al}\leq C(A,A',M).
\end{align*}
\end{thm}

\subsection{Some variational estimates in focusing case}

In the focusing case, the ground state plays an important role.
Consider the fractional elliptic equation
\begin{equation}\label{fellip}
(-\Delta)^{\alpha} W - \abs{W}^{\frac{4\alpha}{N-2\alpha}}W = 0.
\end{equation}
By the work of Lieb \cite{Lieb}, it was known that: if $0<\al<N/2$,
then \eqref{fellip} has a solution in $\dot H^\al$
$$W(x)= C_1(n,\alpha) \left(\frac{1}{1+C_2(n,\alpha)|x|^2}\right)^{\frac{N-2\alpha}2}$$
for some $C_1,C_2>0$. It arises in the study of the best constant
for Hardy-Littlewood-Sobolev inequalities. The classification of
positive regular solutions for \eqref{fellip} was studied in
\cite{CLO}. We also have the following characterization of $W$ (see
\cite{Lieb}, \cite{cotsiolis}): $W$ attains the best constant $C_N$
in the Sobolev embedding inequality:
\begin{equation}\label{Sobem}
\norm{u}_{L^{\frac{2N}{N-2\alpha}}}\le C_N\norm{D^{\alpha} u}_{L^2}.
\end{equation}
Moreover, if $0\ne u\in \dot H^\al$ verifies
$\norm{u}_{L^{\frac{2N}{N-2\alpha}}} = C_N \norm{D^{\alpha}
u}_{L^2}$, then
$u=W_{\theta_0,x_0,\lambda_0}:=e^{i\theta_0}\lambda_0^{(N-2\alpha)/2}W(\lambda_0(x-x_0))$
for some $\theta_0\in [-\pi,\pi]$, $\lambda_0>0$, $x_0\in \R^N$.

$W$ is a stationary solution to \eqref{eq:fNLS} when $\mu=-1$. By
the equation \eqref{fellip}, we have $\int \abso{D^{\alpha} W}^2 =
\int\abso{W}^{2^*}$.  Also, \eqref{Sobem} yields
$C_N^2\int\abso{D^{\alpha} W}^2 =
\paren{\int\abso{W}^{2^*}}^{(N-2\alpha)/N}$, so that
$C_N^2\int\abso{D^{\alpha} W}^2 = \paren{\int\abso{D^{\alpha}
W}^2}^{\frac{N-2\alpha}N}$.  Hence,
$$\int \abso{D^{\alpha} W}^2 = \frac1{C_N^{N/\alpha}} \mbox{ \ \  and \ \ } E_{\mu}(W) = \paren{\frac{1}{2} +\mu \frac{1}{2^*}}\int\abso{D^{\alpha} W}^2,$$
which is  $\frac{\alpha}{N}\frac{1}{C_N^{N/\alpha}}$ in the focusing
case. For simplicity, we write $E_\pm(u)=E_{\pm 1}(u)$.

With the variational properties, we can follow Kenig-Merle's
argument with slight change to prove the following lemma. We omit
the proof.

\begin{lemma}
(1) Assume $\alpha\in (\frac{N}{2N-1},1)$, $\norm{\del u}_{ L^2}<
\norm{\del W}_{L^2}$, and $E_-(u)\le (1-\delta_0)E_{-}(W)$ for some
$\delta_0>0$.  Then, there exists $\overline{\delta} =
\overline{\delta}(\delta_0,N)>0$ such that
\begin{equation}\label{3.5}
\int\abso{\del u}^2\le (1-\overline{\delta})\int\abso{\del W}^2
\end{equation}
and
\begin{equation}\label{3.6}
\int\abso{\del u}^2 - \abso{u}^{2^*}\ge
\overline{\delta}\int\abso{\del u}^2.
\end{equation}

(2) Assume $\alpha\in (\frac{N}{2N-1},1)$. Let $u$ be a solution of
\eqref{eq:fNLS} with maximal interval $I$, $\norm{\del u_0}_{ L^2}<
\norm{\del W}_{L^2}$, and $E_-(u_0)\le (1-\delta_0)E_{-}(W)$ for
some $\delta_0>0$. Then, there exists $\overline{\delta} =
\overline{\delta}(\delta_0,N)>0$ such that for $t\in I$
\begin{equation}\label{3.10}
\int\abso{\del u(t)}^2\le (1-\overline{\delta})\int\abso{\del W}^2
\end{equation}
\begin{equation}\label{3.11}
\int\abso{\del u(t)}^2 - \abso{u(t)}^{2^*}\ge
\overline{\delta}\int\abso{\del u(t)}^2
\end{equation}
\begin{equation}\label{3.12}
E_-(u(t))\simeq \int\abso{\del u(t)}^2\simeq \int\abso{\del u_0}^2
\end{equation}
with comparability constants which depend only on $\delta_0$.
\end{lemma}

\section{Proof of Theorem \ref{thm}}
\subsection{Minimal energy non-scattering solution}
Denote $A_+=\infty, A_-=E_-(W)$. For each $0\le a\le A_\pm$, let
\begin{align*}
\K^-(a):=&\{f\in \dot{H}^\alpha_{rad}: E_-(f)<a,\ \norm{D^\alpha f}_2<\norm{D^\alpha W}_2\}\\
\K^+(a):=&\{f\in \dot{H}^\alpha_{rad}: E_+(f)<a\},\\
\cS^\pm(a):=&\sup\{\|u\|_{S_\al(I)} \mid u(0)\in\K^\pm(a),\
\text{$u$ sol. to \eqref{eq:fNLS} with $\pm$}\},
\end{align*}
Let \EQ{
 E_\pm^* :=\sup\{a>0 \mid \cS^\pm(a)<\I\}.}
The small data scattering implies that $E_\pm^*>0$. We will prove  $E_\pm^*= A_\pm$ by contradiction, and thus
finish the proof of Theorem \ref{thm}.

Assume $E_\pm^*<A_\pm$, then we show the existence of a critical
element which is compact modula invariant groups. We have

\begin{lem}[Existence of critical element]\label{lem:crit}
Suppose $E_\pm^*< A_\pm$, then there is a radial solution $u_\pm$ to
\eqref{eq:fNLS}$\pm$ with maximal interval $I_\pm$ satisfying
\[
{E}(u_\pm)=E_\pm^*,\quad \norm{D^\alpha u_-}_2<\norm{D^\alpha W}_2,\quad \|u_\pm\|_{S(I_\pm)}=\I.
\]
\end{lem}

\begin{lem}\label{lem:comp}
Assume $u_\pm$ is as in Lemma \ref{lem:crit} and say that $\norm{u_\pm}_{S(I_\pm\cap (0,\infty))}=\infty$. Then there exists $\lambda(t)\in \R^+$, for $t\in I_\pm \cap (0,\infty)$, such that
\[K=\{v(x,t):v(x,t)=\frac{1}{\lambda(t)^{\frac{N-2\alpha}{2}}}u_\pm(\frac{x}{\lambda(t)},t)\}\]
has the property that $\overline K$ is compact in $\dot{H}^\alpha$. A corresponding conclusion is reached if $\norm{u_\pm}_{S(I_\pm\cap (-\infty,0))}=\infty$.
\end{lem}

The two lemmas above follow in the same way as Kenig-Merle
\cite{KM}, by using stability Theorem and the profile decomposition
given in \cite{CHKL2}. We omit the details.

\subsection{Rigidity Theorem} The main purpose of this section is to
disprove the existence of critical element that was constructed in
the previous section under the assumption $E_\pm^*<A_\pm$ by using
the structure of the equation \eqref{eq:fNLS}. We will rely on the
virial identity.
\begin{lem}[virial identity]\label{lem:virial}
Assume $u$ is a smooth solution to
\eqref{eq:fNLS}. Then
\[\frac{d}{dt}\re \int iu x\cdot \nabla \bar u dx=2\alpha \int |(-\Delta)^{\alpha/2}u|^2dx+\frac{d\mu p}{p+2}\int |u|^{p+2}dx.\]
\end{lem}

Since the virial does not make sense in the energy space, we will use the localization of virial estimates.
In this sequel, we fix $\psi\in C_0^\infty(\R^N)$, $\psi$ radial, $\psi\equiv 1$ for $|x|<1$, $\psi\equiv 0$ for $|x|\geq 2$.
For $R\ges 1$, let $\psi_R(x)=\psi(x/R)$,
$\tilde{\psi}_R(x)=\frac{x}{R}\cdot \nabla\psi(\frac{x}{R})$. We
have
\begin{lem}\label{lem:lvirial} Assume $u$ is solution to \eqref{eq:fNLS}. Then
\begin{align*}
\frac{d}{dt}\re \int iu x\psi_R\cdot \nabla \bar u dx=&2\alpha \int
|D^\alpha u|^2\psi_R dx+\frac{p\mu d}{p+2}\int
|u|^{p+2}\psi_Rdx\\
&+\Re 2\int D^\alpha u[D^\alpha,\psi_R](x\cdot \nabla \bar
u)dx+\frac{p\mu}{p+2}\int
|u|^{p+2}\tilde{\psi}_Rdx\\
&+\Re d\int D^\alpha u[D^\alpha,{\psi}_R]\bar udx+\Re \int D^\alpha
u[D^\alpha,\tilde{\psi}_R]\bar udx,
\end{align*}
where $[\del,f]g=\del(fg)-f\del g$.
\end{lem}

\begin{proof}
Using the equation \eqref{eq:fNLS}, we get from direct computation
that
\begin{align*}
&\frac{d}{dt}\Re \int iu x\psi_R\cdot \nabla \bar u dx \\
=&\Re \int ((-\Delta)^{\alpha}u+\mu |u|^pu)(2x\psi_R\cdot \nabla
\bar u+d\psi_R\bar u+\tilde{\psi}_R\bar u)dx\\
=&\Re \int D^{2\alpha}u2x\psi_R\cdot \nabla \bar udx+\Re \int \mu |u|^pu2x\psi_R\cdot \nabla \bar udx\\
&+\Re \int D^{2\alpha}u(d\psi_R\bar u+\tilde{\psi}_R\bar u)dx+\Re
\int \mu |u|^pu(d\psi_R\bar u+\tilde{\psi}_R\bar
u)dx\\
:=&I+II+III+IV.
\end{align*}
Obviously,
\begin{align*}
IV=d\mu\int |u|^{p+2}\psi_Rdx+\mu\int \tilde{\psi}_R |u|^{p+2}dx.
\end{align*}
Using integration by part, we get
\begin{align*}
III=&\Re \int D^{2\alpha}u(d\psi_R\bar u+\tilde{\psi}_R\bar u)dx\\
=&d\int |D^{\alpha}u|^2\psi_Rdx+\Re d\int D^\alpha u[D^\alpha,{\psi}_R]\bar udx\\
&+\int |D^{\alpha}u|^2\tilde{\psi}_Rdx+\Re \int D^\alpha
u[D^\alpha,\tilde{\psi}_R]\bar udx.
\end{align*}
Similarly,
\begin{align*}
II=&\Re \int \mu |u|^pu 2x\psi_R\cdot \nabla \bar udx=\int \mu |u|^px\psi_R\cdot \nabla (|u|^2)dx\\
=&-\frac{2\mu d}{p+2}\int |u|^{p+2}\psi_Rdx-\frac{2\mu}{p+2}\int
|u|^{p+2}\tilde{\psi}_Rdx
\end{align*}

Now we compute $I$. By Fourier transfrom, it is easy to check
$[D^\alpha,x\cdot \nabla]=\alpha D^\alpha$. Then we have
\begin{align*}
I=&\Re 2\int D^\alpha u\psi_R(x\cdot \nabla D^\alpha \bar u+\alpha
D^\alpha \bar u)dx+\Re 2\int D^\alpha u[D^\alpha,\psi_R](x\cdot
\nabla \bar u)dx\\
=&2\alpha \int |D^\alpha u|^2\psi_R dx-d\int |D^\alpha u|^2\psi_R
dx\\
&-\int |D^\alpha u|^2\tilde{\psi}_R dx+\Re 2\int D^\alpha
u[D^\alpha,\psi_R](x\cdot \nabla \bar u)dx.
\end{align*}
Summing over the four terms, we complete the proof.
\end{proof}

Due to the nonlocal properties of the fractional Schr\"odinger equation, the localization of virial estimates is not very clean. There are many remainder terms. However, all of them can be handled in the energy space. We have

\begin{lem}\label{lem:com}
Assume $0<\alpha\leq 1$, $0<\e<\alpha$ and $R\ges 1$. Then
\begin{align}
\norm{[D^\alpha,\psi_R]f}_{L^2}\les& \norm{g}_{L^{\frac{2N}{N-2\alpha}}(|x|\ges
R^{1-\e})}+R^{-\e \alpha}\norm{D^\alpha f}_{L^2},\label{eq:vier1}\\
\norm{[D^\alpha,\psi_R]x\cdot \nabla f}_{L^2}\les& \norm{D^\alpha f}_{L^2},\label{eq:vier2}\\
\norm{[D^\alpha,\psi_R]x\cdot \nabla f}_{L^2(|x|\les R^{1-\e})}\les&
R^{-\e\alpha/2}\norm{D^\alpha f}_{L^2}+\norm{g}_{L^{\frac{2n}{n-2\alpha}}(|x|\ges
R^{1-\e})},\label{eq:vier3}
\end{align}
where $g=\ft^{-1}(|\hat f|)$.
\end{lem}

The proof of the lemma will be given in the end of this section. Now
we use it to prove the main result of this section:
\begin{thm}\label{thm:rig}
Assume that $u_0^\pm\in \dot{H}^\alpha$ is such that
\[E_\pm(u_0^\pm)<A_\pm,\quad \norm{D^\alpha u_0^-}_2<\norm{D^\alpha W}_2.\]
Let $u_\pm$ be the solution of \eqref{eq:fNLS}$\pm$ with $u_\pm(0)=u_0^\pm$, with maximal interval of existence $I_\pm$. Assume that there exists $\lambda(t)>0$, for $t\in I_\pm\cap [0,\infty)$, with the property that
\[K=\{v(x,t):v(x,t)=\frac{1}{\lambda(t)}u_\pm(\frac{x}{\lambda(t)},t)\}\]
is precompact in $\dot{H}^\alpha$.
 Then we must have $T_+(u_0)=\infty$, $u_0\equiv 0$.
\end{thm}

\begin{proof}[Proof of Theorem \ref{thm:rig}]
We only prove the focusing case, since the defocusing case follows
in a similar way.  Assume $I_-=(-T_-,T_+)$. It suffices to prove
this theorem under the assumption that $\lambda(t)\geq A_0$ for some
$A_0>0$ for all $t$, since the general case follows similarly as in
\cite{KM}. The proof splits in two cases.

{\bf Case 1. $T_+(u_0)<+\infty$}

With the same proof as in \cite{KM}, we have
$\lambda(t)\rightarrow\infty$ as $t\uparrow T_+(u_0)$. We define
\[y_R(t)=\int |u(x,t)|^2\psi_R(x)dx, \quad t\in [0,T_+).\]
Then we have
\begin{align*}
y_R'(t)=&-2\im \int D^{2\alpha}(u)\cdot \bar u\psi_R(x)dx\\
=&-2\im \int D^{\alpha}u\cdot [D^\alpha(\bar u\psi_R)- \bar u\cdot
D^\alpha\psi_R]dx-2\im \int D^{\alpha}u\cdot \bar u\cdot
D^\alpha\psi_Rdx.
\end{align*}
By the commutator estimates $\norm{D^\al(fg)-fD^\al g}_2\les
\norm{D^\al g}_2\norm{g}_\infty$, $\al\in (0,1)$, we get
\begin{align*}
|y_R'(t)|\les& \norm{D^{\alpha}u}_2\norm{D^\alpha(\bar u\psi_R)- \bar u\cdot D^\alpha\psi_R}_2+\norm{D^{\alpha}u}_2\norm{u}_{\frac{2n}{n-2\alpha}}\norm{D^\alpha\psi_R}_{\frac{n}{\alpha}}\\
\les& \norm{D^{\alpha}u}_2^2.
\end{align*}

Next, we show: for all $R>0$,
\begin{align}\label{eq:0limit}
\int_{|x|<R}|u(x,t)|^2dx\to 0, \quad \mbox{as } t\to T_+(u_0).
\end{align}
In fact, $u(y,t)=\lambda(t)^\frac{N-2\alpha}{2}v(\lambda(t)y,t)$ so
that
\begin{align*}
&\int_{|x|<R}|u(x,t)|^2dx\\
=&\lambda(t)^{-2\alpha}\int_{|y|<R\lambda(t)}|v(y,t)|^2dy\\
=&\lambda(t)^{-2\alpha}\int_{|y|<\varepsilon R\lambda(t)}|v(y,t)|^2dy+\lambda(t)^{-2\alpha}\int_{\varepsilon R\lambda(t)<|y|<R\lambda(t)}|v(y,t)|^2dy\\
:=&I+II.
\end{align*}
By H\"older and Sobolev, we have $$I\lesssim
\lambda(t)^{-2\alpha}\|v\|_{L^\frac{2N}{N-2\alpha}}^2(\varepsilon
R\lambda(t))^{2\alpha}\lesssim (\varepsilon R)^{2\alpha}\|D^\alpha
W\|_2^2,$$ while $$II\lesssim
\lambda(t)^{-2\alpha}(R\lambda(t))^{2\alpha}\|v\|_{L^\frac{2N}{N-2\alpha}(|x|\geq
\varepsilon R\lambda(t))}^2\rightarrow0,\ \mathrm{as}\ t\rightarrow
T_+(u_0).$$ Thus \eqref{eq:0limit} follows.

Therefore, we have
\begin{equation*}
|y_R(0)-y_R(T_+(u_0))|\lesssim T_+(u_0)\|D^\alpha W\|_2^2,
\end{equation*}
which implies $$y_R(0)\lesssim T_+(u_0)\|D^\alpha W\|_2^2.$$ Then
letting $R\rightarrow\infty$, we obtain that $u_0\in L^2(\Bbb R^N)$.
Arguing as before, $$|y_R(t)-y_R(T_+(u_0))|\lesssim
(T_+(u_0)-t)\|D^\alpha W\|_2^2.$$ So $$|y_R(t)|\lesssim
(T_+(u_0)-t)\|D^\alpha W\|_2^2.$$ Letting $R\rightarrow\infty$, we
see that $$\|u(t)\|_2^2\lesssim (T_+(u_0)-t)\|D^\alpha W\|_2^2$$ and
so by the conservation of the $L^2$ norm $\|u_0\|_2=\|u(t)\|_2\to 0,
 t\to T_+(u_0)$. But that $u\equiv0$ contradicting
$T_+(u_0)<+\infty$.

{\bf Case 2. $T_+(u_0)=+\infty$}

In this case we use the localized virial identity. Let $u(y,t)=\lambda(t)^\frac{N-2\alpha}{2}v(\lambda(t)y,t)$, then
\begin{align*}
\int_{|y|>R(\varepsilon)}|D^\alpha u(y,t)|^2dy=&\int_{|y|>R(\varepsilon)}\lambda(t)^N|D^\alpha v(\lambda(t)y,t)|^2dy\\
=&\int_{|z|>\lambda(t)R(\varepsilon)}|D^\alpha v(z,t)|^2dx\\
\leq&\int_{|z|\geq A_0R(\varepsilon)}|D^\alpha v(z,t)|^2dz\\
\lesssim& \ \varepsilon.\ (\mathrm{by\ the\ precompactness\ of\ K})
\end{align*}
By similar arguments, we have for any $\varepsilon>0$, there exists $R(\varepsilon)$ such that \begin{equation}\label{compu}\int_{|x|>R(\varepsilon)}\Big(|D^\alpha u(x,t)|^2+|u(x,t)|^\frac{2N}{N-2\alpha}+\frac{|u(x,t)|^2}{|x|^{2\alpha}}\Big)dx<\varepsilon\end{equation}
Let $\tilde u=\ft^{-1}_x |\ft u(\xi,t)|$. By Plancherel theorem we know $\tilde u$ has the same compactness as $u$. Thus we have: for each $\e>0$, there exists $R(\e)>0$ such that, for all $t\in [0,\infty)$, we have
\begin{equation}\label{comp}\int_{|x|>R(\e)}\Big(|D^\alpha \tilde u(x,t)|^2+|\tilde u(x,t)|^\frac{2N}{N-2\alpha}+\frac{|\tilde u(x,t)|^2}{|x|^{2\alpha}}\Big)dx<\e.\end{equation}

Next, we consider
\[I_R(t)=\re \int iu x\psi_R\cdot \nabla \bar u dx.\]
By Sobolev multiplication laws, we have
\begin{align*}
|I_R(t)|\les &\norm{D^{1-\alpha}(u x\psi_R)}_2\cdot \norm{D^{\alpha-1}\nabla \bar u}_2\\
\les &\norm{D^{\alpha}u}_2\cdot \norm{D^{1+\frac{d}{2}-2\alpha}(x\psi_R)}_2\cdot\norm{D^{\alpha}u}_2\les R^{2\alpha} \cdot \norm{D^{\alpha}u_0}_2^2.
\end{align*}
On the other hand, by Lemma \ref{lem:lvirial} and Lemma \ref{lem:com}, we have
\begin{align}
I_R'(t)=&2\alpha\int|D^\alpha u|^2dx-2\alpha\int|u|^\frac{2N}{N-2\alpha}dx\label{eq:v1}\\
&+2\alpha\int(|D^\alpha u|^2-|u|^\frac{2N}{N-2\alpha})(\psi_R-1)dx-\frac{2\alpha}{N}\int|u|^\frac{2N}{N-2\alpha}\tilde{\psi}_Rdx\label{eq:v2}\\
&+2\re \int D^\alpha u[D^\alpha, \psi_R](x\cdot\nabla\bar{u})dx\label{eq:v3}\\
&+d\re \int D^\alpha u[D^\alpha, \psi_R]\bar{u}dx+\re\int D^\alpha u[D^\alpha, \tilde{\psi}_R]\tilde{u}dx.\label{eq:v4}
\end{align}

By the variational estimates, we have$$(\ref{eq:v1})\geq C_\delta
\|D^\alpha u_0\|^2_2.$$ If $u_0\ne 0$, then fix $0<\e\ll \|D^\alpha
u_0\|^2_2$. For (\ref{eq:v2}), by (\ref{compu}) we get that
$$(\ref{eq:v2})\les \varepsilon$$ for $R$ sufficiently large. The
first term of (\ref{eq:v3}) can be estimated as follows
\begin{align*}
&|(\ref{eq:v3})|\\
\lesssim& |\int_{|x|\lesssim R^{1-\varepsilon}}D^\alpha u[D^\alpha, \psi_R](x\cdot\nabla\bar{u})dx|+|\int_{|x|\gtrsim R^{1-\varepsilon}}D^\alpha u[D^\alpha, \psi_R](x\cdot\nabla\bar{u})dx|\\
\lesssim&\|D^\alpha u\|_2\|[D^\alpha, \psi_R](x\cdot\nabla\bar{u})\|_{L^2(|x|\lesssim R^{1-\varepsilon})}+\|D^\alpha u\|_{L^2(|x|\gtrsim R^{1-\varepsilon})}\|[D^\alpha, \psi_R](x\cdot\nabla\bar{u})\|_2\\
\lesssim&R^{-\frac{\varepsilon\alpha}{2}}\|D^\alpha u\|_2+\|\tilde{u}\|_{L^\frac{2N}{N-2\alpha}(|x|\gtrsim R^{1-\varepsilon})}+\|D^\alpha u\|_{L^2(|x|\gtrsim R^{1-\varepsilon})}\|D^\alpha u\|_{L^2}
\end{align*}
where the last inequality follows from Lemma 5.3. Therefore,
$(\ref{eq:v3})\lesssim \varepsilon$ if $R$ is sufficiently large.
The smallness of (\ref{eq:v4}) can be obtained similarly. Thus
$$|I_R'(t)|\gtrsim \int|D^\alpha u_0|^2.$$
Integrating in $t$, we get $I_R(t)-I_R(0)\gtrsim t\int|D^\alpha
u_0|^2$, but we also have $|I_R(t)-I_R(0)|\lesssim R^2\int|D^\alpha
u_0|^2$, which is a contradiction for $t$ large. Thus $u_0\equiv 0$
and the theorem is proved.
\end{proof}

In the end, we give the proof of Lemma \ref{lem:com}.

\begin{proof}[Proof of Lemma \ref{lem:com}]
First we show \eqref{eq:vier1}. Using Fourier transform, we have
\begin{align*}
|\ft([D^\alpha,\psi_R]f)(\xi)|\les&
|\int_{\xi=\xi_1+\xi_2}(|\xi_1+\xi_2|^\alpha-|\xi_2|^\alpha)\wh{\psi_R}(\xi_1)\wh{f}(\xi_2)|\\
\les&\int_{\xi=\xi_1+\xi_2} |\xi_1|^\alpha
|\wh{\psi_R}(\xi_1)|\cdot|\wh{f}(\xi_2)|.
\end{align*}
Then we get
\begin{align*}
\norm{[D^\alpha,\psi_R]f}_2\les& \norm{\ft^{-1}(|\xi_1|^\alpha
|\wh{\psi_R}(\xi_1)|)\cdot g}_2\\
\les& \norm{\ft^{-1}(|\xi_1|^\alpha
|\wh{\psi_R}(\xi_1)|)}_{\frac{n}{\alpha}}\cdot
\norm{g}_{L^{\frac{2n}{n-2\alpha}}(|x|\ges
R^{1-\e})}\\
&\quad+\norm{\ft^{-1}(|\xi_1|^\alpha
|\wh{\psi_R}(\xi_1)|)}_{L^{\frac{n}{\alpha}}(|x|\les R^{1-\e})}\cdot\norm{g}_{L^{\frac{2n}{n-2\alpha}}}\\
\les& \norm{g}_{L^{\frac{2n}{n-2\alpha}}(|x|\ges R^{1-\e})}+R^{-\e
\alpha}\norm{D^\alpha f}_{L^2}
\end{align*}
where in the last inequality we used the fact that
$\norm{\ft^{-1}(|\xi_1|^\alpha
|\wh{\psi_R}(\xi_1)|)}_{\frac{n}{\alpha}}\leq C$,
$|\ft^{-1}(|\xi_1|^\alpha |\wh{\psi_R}(\xi_1)|)|\les R^{-\alpha}$
and the Sobolev embedding.

Next, we prove \eqref{eq:vier2}. Direct computations show that
\begin{align*}
&\ft([D^\alpha,\psi_R]x\cdot\nabla f)(\xi) \\
=&-
\int(|\xi|^\alpha-|\xi_2|^\alpha)\wh{\psi_R}(\xi-\xi_2)\nabla_{\xi_2}\cdot(\xi_2\wh{f}(\xi_2))d\xi_2\\
=&\int_{\xi=\xi_1+\xi_2}-\alpha|\xi_2|^{\alpha}\wh{\psi_R}(\xi_1)\wh{f}(\xi_2)+i(|\xi_1+\xi_2|^\alpha-|\xi_2|^\alpha)\wh{x\psi_R}(\xi_1)\cdot
\xi_2\wh{f}(\xi_2).
\end{align*}
Thus we get
\begin{align*}
|\ft([D^\alpha,\psi_R]x\cdot\nabla
f)(\xi)|\les&\int_{\xi=\xi_1+\xi_2}
|\xi_2|^{\alpha}(|\wh{\psi_R}(\xi_1)|+|\wh{x\psi_R}(\xi_1)|\cdot|\xi_1|)\cdot|\wh{f}(\xi_2)|
\end{align*}
and then by Plancherel's equality
\begin{align*}
\norm{[D^\alpha,\psi_R]x\cdot\nabla f}_{L^2}\les& \norm{D^\alpha
f}_{L^2}.
\end{align*}

Finally, we prove \eqref{eq:vier3}. We have
\begin{align*}
&\ft([D^\alpha,\psi_R]x\cdot\nabla f)(\xi)\\
=&\int_{\xi=\xi_1+\xi_2,|\xi_1|\ll
|\xi_2|}-\alpha|\xi_2|^{\alpha}\wh{\psi_R}(\xi_1)\wh{f}(\xi_2)+i(|\xi_1+\xi_2|^\alpha-|\xi_2|^\alpha)\wh{x\psi_R}(\xi_1)\cdot
\xi_2\wh{f}(\xi_2)\\
&+\int_{\xi=\xi_1+\xi_2,|\xi_1|\ges
|\xi_2|}-\alpha|\xi_2|^{\alpha}\wh{\psi_R}(\xi_1)\wh{f}(\xi_2)+i(|\xi_1+\xi_2|^\alpha-|\xi_2|^\alpha)\wh{x\psi_R}(\xi_1)\cdot
\xi_2\wh{f}(\xi_2) \\
:=&\ft[M(f)]+\ft[R(f)].
\end{align*}
As before, we have
\begin{align*}
|\ft[R(f)](\xi)|\les&\int_{\xi=\xi_1+\xi_2}
|\xi_1|^{\alpha}(|\wh{\psi_R}(\xi_1)|+|\wh{x\psi_R}(\xi_1)|\cdot|\xi_1|)\cdot|\wh{f}(\xi_2)|
\end{align*}
and then as \eqref{eq:vier1} we get
\[\norm{Rf}_{2}\les \norm{g}_{L^{\frac{2n}{n-2\alpha}}(|x|\ges
R^{1-\e})}+R^{-\e \alpha}\norm{D^\alpha f}_{L^2}.\]

To estimate $M(f)$, we need to exploit a cancelation. Since
$|\xi_1|\ll |\xi_2|$, by fundamental theorem of calculus we have
\[|\xi_1+\xi_2|^\alpha-|\xi_2|^\alpha=\int_0^1\frac{d}{dt}|t\xi_1+\xi_2|^\alpha dt
=\int_0^1\alpha
|t\xi_1+\xi_2|^{\alpha-1}\frac{t\xi_1+\xi_2}{|t\xi_1+\xi_2|}dt\cdot
\xi_1.\] Thus we get
\begin{align}
\ft[M(f)]=&\int_{|\xi_1|\ll
|\xi_2|}-\alpha|\xi_2|^{\alpha}\wh{\psi_R}(\xi_1)\wh{f}(\xi_2)\nonumber\\
&+\int_{|\xi_1|\ll |\xi_2|}\int_0^1\alpha
|t\xi_1+\xi_2|^{\alpha-1}\frac{t\xi_1+\xi_2}{|t\xi_1+\xi_2|}dt\cdot
i\xi_1\wh{x\psi_R}(\xi_1)\cdot \xi_2\wh{f}(\xi_2)\label{eq:Mf}
\end{align}
Denote $\xi_s=(\xi_{s,1},\cdots,\xi_{s,n}), s=1,2$, then the second
term equals to
\begin{align*}
&\int_{|\xi_1|\ll |\xi_2|}\sum_{j,k=1}^n\int_0^1\alpha
|t\xi_1+\xi_2|^{\alpha-1}\frac{t\xi_{1,k}+\xi_{2,k}}{|t\xi_1+\xi_2|}dt\cdot
i\xi_{1,k}\wh{x_j\psi_R}(\xi_1)\cdot \xi_{2,j}\wh{f}(\xi_2)\\
=&\int_{|\xi_1|\ll |\xi_2|}\int_0^1\alpha
|t\xi_1+\xi_2|^{\alpha-1}\frac{t\xi_1+\xi_2}{|t\xi_1+\xi_2|}dt\cdot
\xi_2\wh{\psi_R}(\xi_1)\wh{f}(\xi_2)\\
&+\int_{|\xi_1|\ll |\xi_2|}\int_0^1\alpha
|t\xi_1+\xi_2|^{\alpha-1}\frac{t\xi_1+\xi_2}{|t\xi_1+\xi_2|}dt\cdot
 \wh{x\otimes \nabla {\psi}_R}(\xi_1)\cdot \xi_2\wh{f}(\xi_2).
\end{align*}
Thus, we get
\begin{align*}
\ft[M(f)]=&\int_{|\xi_1|\ll |\xi_2|}(\int_0^1\alpha
|t\xi_1+\xi_2|^{\alpha-1}\frac{t\xi_1+\xi_2}{|t\xi_1+\xi_2|}dt\cdot
\xi_2-\alpha|\xi_2|^{\alpha})\wh{\psi_R}(\xi_1)\wh{f}(\xi_2)\\
&+\int_{|\xi_1|\ll |\xi_2|}i\int_0^1\alpha
|t\xi_1+\xi_2|^{\alpha-1}\frac{t\xi_1+\xi_2}{|t\xi_1+\xi_2|}dt\cdot
|\xi_2|^{-\alpha}\xi_2 \wh{\tilde{\psi}_R}(\xi_1)\wh{D^\alpha
f}(\xi_2)\\
=&\ft[I]+\ft[II].
\end{align*}
For I, by mean value formula, we have
\[|I|\les \int_{|\xi_1|\ll
|\xi_2|}\alpha^2|\xi_2|^{\alpha-1}|\xi_1|\cdot|\wh{\psi_R}(\xi_1)|\cdot|\wh{f}(\xi_2)|\]
and then
\[\norm{I}_2\les R^{-\alpha}\norm{f}_2.\]
For II, we see
\[II=\int K(x-y_1,x-y_2)\tilde{\psi}_R(y_1)D^\alpha f(y_2)dy_1dy_2\]
where $K$ is the kernel for the bilinear multiplier
\[K(x,y)=\int e^{i(x\xi_1+y\xi_2)}m(\xi_1,\xi_2)d\xi_1d\xi_2\]
with the symbol
\[m(\xi_1,\xi_2)=\int_0^1\alpha
|t\xi_1+\xi_2|^{\alpha-1}\frac{t\xi_1+\xi_2}{|t\xi_1+\xi_2|}dt\cdot
|\xi_2|^{-\alpha}\xi_2 \cdot 1_{|\xi_1|\ll |\xi_2|}.\] It is easy to
see from direct computations that $m$ satisfy the Coifman-Meyer's
H\"ormander-type condition, and then
\[|K(x-y_1,x-y_2)|\les (|x-y_1|+|x-y_2|)^{-2n}.\]
If $|y_1|\sim R, |x|\les R^{1-\e}$, then $|K(x-y_1,x-y_2)|\les
R^{-2n}$. Thus we get
\[\norm{II}_{L^2(|x|\les R^{1-\e})}\les R^{-\e/2}\norm{D^\alpha f}_2.\]
Therefore, the lemma is proved.
\end{proof}
%
%\section*{Appendix: collections of some proofs}
%
%\begin{proof}[Proof of Theorem \ref{thm}]
%
%\end{proof}

\end{document}